\documentclass[10pt]{amsart}
\usepackage{amsmath}
\usepackage{amssymb}
\usepackage{amsfonts}
\usepackage{mathrsfs}
\usepackage[sans]{dsfont} % use \mathds{1} for indicator function. Erase [sans] for a different style
\usepackage{enumerate}

\numberwithin{equation}{section}       % Number formulas within sections
\numberwithin{figure}{section}       % Number figures within sections
\theoremstyle{plain}

\newtheorem{prop}{Proposition}[section]

\newtheorem{lemm}[prop]{Lemma}

\newtheorem{theoalph}{Theorem}

\newtheorem{propalph}[theoalph]{Proposition}

\theoremstyle{definition}

\theoremstyle{remark}

\newtheoremstyle{citing}% name
  {3pt}%      Space above, empty = `usual value'
  {3pt}%      Space below
  {\itshape}% Body font
  {}%         Indent amount (empty = no indent, \parindent = para indent)
  {\bfseries}% Thm head font
  {.}%        Punctuation after thm head
  {.5em}%     Space after thm head: " " = normal interword space;
  {\thmnote{#3}}% Thm head spec

\theoremstyle{citing}
\newtheorem*{generic}{}% all text supplied in the note

\newcommand{\partn}[1]{{\smallskip \noindent \textbf{#1.}}}% for parts of a proof
\DeclareMathAlphabet{\mathpzc}{OT1}{pzc}{m}{it} % Zapf Chancery math alphabet

\newcommand{\C}{\mathbb{C}}

\newcommand{\N}{\mathbb{N}}

\newcommand{\R}{\mathbb{R}}

\newcommand{\cO}{\mathcal{O}}

\newcommand{\cT}{\mathcal{T}}

\newcommand{\sA}{\mathscr{A}}

\newcommand{\sM}{\mathscr{M}}

\newcommand{\sS}{\mathscr{S}}

\newcommand{\hJ}{\widehat{J}}
\newcommand{\hK}{\widehat{K}}

\newcommand{\hW}{\widehat{W}}

\newcommand{\teta}{\widetilde{\teta}}

\renewcommand{\=}{ : = }

\DeclareMathOperator{\diam}{diam}
\DeclareMathOperator{\Crit}{Crit} % Set of critical points
\DeclareMathOperator{\per}{per}

\newcommand{\chiinf}{\chi_{\inf}}
\newcommand{\chiinff}{\chi_{\inf}(f)}
\newcommand{\chiper}{\chi_{\per}}

\newcommand{\whm}{\widehat{m}}

\newcommand{\fat}{f_{a, \tau}}
\newcommand{\Kat}{K_{a, \tau}}
\newcommand{\Vat}{V_{a, \tau}}
\newcommand{\vat}{v_{a, \tau}}
\newcommand{\lambdaat}{\lambda_{a, \tau}}

\begin{document}

\title[Asymptotic expansion of maps with disconnected Julia set]{On the asymptotic expansion of maps with disconnected Julia set}
\author{Juan Rivera-Letelier}
% \thanks{}
\address{Juan Rivera-Letelier, Facultad de Matem{\'a}ticas, Pontificia Universidad Cat{\'o}lica de Chile, Avenida Vicu{\~n}a Mackenna~4860, Santiago, Chile}
\email{riveraletelier@mat.puc.cl}

\begin{abstract}
We study the asymptotic expansion of smooth one-dimensional maps.
We give an example of an interval map for which the optimal shrinking of components exponential rate is not attained for any neighborhood of a certain fixed point in the boundary of a periodic Fatou component.
As a complement, we prove a general result asserting that when this happens the components do shrink exponentially, although the rate is not the optimal one.
Finally, we give an example of a polynomial with real coefficients, such that all its critical points in the complex plane are real, and such that its asymptotic expansion as a complex map is strictly smaller than its asymptotic expansion as a real map.
\end{abstract}
% \subjclass[2000]{S99}
% \keywords{Key words}

\maketitle

\section{Introduction}

This note is concerned with non-uniform hyperbolicity notions for smooth one-dimensional maps.
The Fatou and Julia sets of such a map are important in what follows.
See~\cite{CarGam93,Mil06} for expositions in the complex case.
For a non-degenerate smooth interval map we consider the theory of Fatou and Julia sets developed by Martens, de Melo, and van Strien in~\cite{MardMevSt92}, see also~\cite{dMevSt93}.

For a smooth one-dimensional map, several ways to measure the asymptotic expansion coincide, see~\cite{NowSan98,PrzRivSmi03,Riv1204}.
This note is mainly focused on the asymptotic expansion of interval maps with disconnected Julia set.
We show that to obtain the optimal shrinking of components exponential rate, in some cases it is necessary to exclude intervals that intersect a periodic Fatou component (Proposition~\ref{p:exceptional ESC}).
This shows that a seemingly technical hypothesis of~\cite[Main Theorem']{Riv1204} is in fact necessary.
We also show a qualitative result that holds without this hypothesis (Theorem~\ref{t:qualitative asymptotic expansion}).

We also study the relation between the asymptotic expansions of the real and the complex map defined by a polynomial with real coefficients.
We give an example of a polynomial with real coefficients, such that all its critical points in~$\C$ are real, and such that its asymptotic expansion as a map of~$\C$ is different from its asymptotic expansion as a map of~$\R$ (Proposition~\ref{p:real versus complex}).

We proceed to state our results more precisely.
\subsection{Exponential rate of shrinking of components}
\label{ss:exponential shrinking rate}
Let~$I$ be a compact interval of~$\R$.
A non-injective smooth map~$f : I \to I$ is \emph{non-degenerate}, if the set of points at which the derivative of~$f$ vanishes is finite and if at each of these points a higher order derivative of~$f$ is non-zero.
In what follows, a connected component of the Fatou set of~$f$ is called \emph{Fatou component}.

For a non-degenerate smooth map~$f : I \to I$, the number
$$ \chiinff 
\= 
\left\{ \int |Df| d \mu : \mu \text{ probability measure on~$J(f)$ invariant by~$f$} \right\} $$
is a measure of the asymptotic expansion of~$f$.
The condition~$\chiinff > 0$ can be seen as a strong form of non-uniform hyperbolicity in the sense of Pesin and has strong implications for the dynamics of~$f$: It implies the existence of an exponentially mixing absolutely continuous invariant measure, see~\cite[Corollary~A]{Riv1204}, and also~\cite[Theorem~A]{NowSan98} for maps with one critical point.

When~$f$ is topologically exact on~$J(f)$, several numbers that measure the asymptotic expansion of~$f$ coincide with~$\chiinff$.
For example, for every sufficiently small interval~$J$ contained in~$I$ that is not a neighborhood of a periodic point in the boundary of a Fatou component, we have
\begin{multline}
  \label{e:optimal ESC}
 \lim_{n \to + \infty} \frac{1}{n} \ln \max \left\{ |W| : W \text{ connected component of~$f^{-n}(J)$} \right\}
\\ =
- \chiinff,
\end{multline}
see~\cite[Main Theorem']{Riv1204}.
We note that the hypothesis that~$J$ is not a neighborhood of a periodic point in the boundary of a Fatou component is only required in~\cite[Main Theorem']{Riv1204} in the case where~$\chiinff > 0$ and where~$J(f)$ is not an interval.
Our first result is that~\eqref{e:optimal ESC} can fail if this hypothesis is not satisfied.
\begin{propalph}
  \label{p:exceptional ESC}
There is a non-degenerate smooth map~$f : I \to I$ that is topologically exact on its Julia set, and that satisfies~$\chiinff > 0$ and the following property: There is a fixed point~$p$ of~$f$ in the boundary of a Fatou component of~$f$, such that for every~$\delta > 0$ we have
\begin{multline*}
  \label{e:non-optimal ESC}
 \limsup_{n \to + \infty} \frac{1}{n} \ln \max \left\{ |W| : W \text{ connected component of~$f^{-n}(B(p, \delta))$} \right\}
\\ >
- \chiinff.
\end{multline*}
\end{propalph}

We also show that a weak version of~\eqref{e:optimal ESC} does hold for every sufficiently small interval~$J$ intersecting~$J(f)$.
\begin{theoalph}
\label{t:qualitative asymptotic expansion}
Let~$f$ be a non-degenerate smooth interval map that is topologically exact on~$J(f)$ and such that~$\chiinff > 0$.
Then there are~$\delta > 0$ and~$\lambda > 1$ such that for every point~$x$ in~$J(f)$, for every integer~$n \ge 1$, and every pull-back~$W$ of~$B(x, \delta)$ by~$f^n$, we have~$|W| \le \lambda^{-n}$.
\end{theoalph}
By~\cite[Main Theorem']{Riv1204} the constant~$\lambda$ in Theorem~\ref{t:qualitative asymptotic expansion} is less than or equal to~$\exp(\chiinff)$.
On the other hand, Proposition~\ref{p:exceptional ESC} shows that in general~$\lambda$ cannot be taken arbitrarily close to~$\exp(\chiinff)$ if~$J$ is a neighborhood of a periodic point in the boundary of a Fatou component.

\subsection{Real versus complex asymptotic expansion}
\label{ss:real versus complex}
For a complex rational map~$f$ of degree at least~$2$, the limit~\eqref{e:optimal ESC}, with~$J$ replaced by any sufficiently small ball centered at a point in~$J(f)$, is shown in~\cite[Main Theorem]{PrzRivSmi03}.

Note that a polynomial of degree at least~$2$ with real coefficients~$f$ defines a map of~$\R$ and a map of~$\C$.
In the following proposition we show that, even if all the critical points of~$f$ in~$\C$ are real, the asymptotic expansion of~$f$ as a map of~$\C$ might be strictly smaller than its asymptotic expansion as a map of~$\R$.
\begin{propalph}
  \label{p:real versus complex}
There is a polynomial~$P$ of degree~$4$ with real coefficients such that all its critical points in~$\C$ are real and such that, if we denote by~$\sM^{\R}$ (resp.~$\sM^{\C}$) the space of all Borel probability measures on~$\R$ (resp.~$\C$) that are invariant by~$P$, then
$$ 0
<
\inf \left\{ \int |D P| d \mu : \mu \in \sM^{\C} \right\}
<
\inf \left\{ \int |D P| d \mu : \mu \in \sM^{\R} \right\}. $$
\end{propalph}

\subsection{Organization}
\label{ss:organization}
The proofs of Propositions~\ref{p:exceptional ESC} and~\ref{p:real versus complex} occupy~\S\S\ref{s:the family}--\ref{s:proof of the Main Lemma}.
In~\S\ref{s:the family} we introduce a~$2$ parameter family of polynomials of degree~$4$ and reduce the proofs of Propositions~\ref{p:exceptional ESC} and~\ref{p:real versus complex} to the existence of a member of this family with properties analogous to those of the map~$f$ in Proposition~\ref{p:exceptional ESC} (Main Lemma).
In~\S\ref{s:mapping properties} we consider some general mapping properties of the maps introduced in~\S\ref{s:the family}, and in~\S\ref{s:combinatorial type} we select parameters for which the corresponding maps have some special combinatorial properties.
In~\S\ref{s:estimates} we choose the map used to prove the Main Lemma and make the main estimates used in the proof.
The proof of the Main Lemma is given in~\S\ref{s:proof of the Main Lemma}.

The proof of~Theorem~\ref{t:qualitative asymptotic expansion} is in~\S\ref{s:proof of ESC revisited} and it is independent of the previous sections.
Some general properties of non-degenerate smooth maps that are used in the proof are gathered in Appendix~\ref{s:general properties}.

\subsection{Acknowledgments}
\label{ss:acknowledgments}
This article was written while the author was visiting Brown University and the Institute for Computational and Experimental Research in Mathematics (ICERM), and during the workshop ``Non-uniformly hyperbolic and neutral one-dimensional dynamics,'' held at the Institute for Mathematical Sciences of the National University of Singapore.
The author thanks all of these institutions for the optimal working conditions provided, and acknowledges partial support from FONDECYT grant 1100922.

\section{A~$2$ parameter family of polynomials}
\label{s:the family}
After introducing a~$2$ parameter family of polynomials, we state the Main Lemma and deduce Propositions~\ref{p:exceptional ESC} and~\ref{p:real versus complex} from it.

For real parameters~$a$ and~$\tau$, consider the real polynomial
$$ \fat(x)
=
1 - \tau + a x^2 - (a + 2 - \tau) x^4. $$
Note that~$\fat(1) = \fat(-1) = -1$ and that for~$a \neq 0$ the point~$x = 0$ is a critical point of~$\fat$.

The maximal invariant set~$\Kat$ of~$\fat$ in~$[-1, 1]$ is important in what follows.
When~$a \ge 10$ and~$\tau \le 2$, all the critical points of~$\fat$ in~$\C$ are real, every point in~$\R \setminus [-1, 1]$ converges to~$- \infty$ under forward iteration by~$\fat$, and~$x = 0$ is the only critical point of~$\fat$ in~$\Kat$, see~\S\ref{s:mapping properties}.
So in this case~$\Kat$ is contained in~$[-1, 1]$.

The purpose of this section is to deduce Propositions~\ref{p:exceptional ESC} and~\ref{p:real versus complex} from the following lemma.
For~$a \ge 10$ and~$\tau \le 2$, put
\begin{multline*}
\chiper(\fat)
\\ \=
\inf \left\{ \frac{1}{n} \ln |D\fat^n(p)| : n \ge 1, p \text{ periodic point of~$\fat$ of period } n \right\}.  
\end{multline*}
\begin{generic}[Main Lemma]
There are parameters~$a \ge 10$ and~$\tau \le 2$ such that~$\fat$ is topologically exact on~$\Kat$, and such that the following properties hold: $\chiper(\fat) > 0$, the critical point~$x = 0$ is in~$\Kat$, $\fat$ satisfies the Collet-Eckmann condition:
\begin{equation}
\label{e:CE condition}
\liminf_{n \to + \infty} \frac{1}{n} \ln |D \fat^n(\fat(0))| > 0,
\end{equation}
and for every~$\delta > 0$ we have:
\begin{multline*}
\limsup_{m \to + \infty} \frac{1}{m} \ln \left\{ |W| : W \text{ connected component of } \fat^{- m}\left( (-1 - \delta, -1] \right) \right\}
\\ >
- \chiper(\fat).
\end{multline*}
\end{generic}
The proof of the Main Lemma occupies~\S\S\ref{s:mapping properties}--\ref{s:proof of the Main Lemma}.

For the parameters~$a$ and~$\tau$ given by the Main Lemma there is no compact interval that is invariant by~$\fat$.
So we need to modify~$\fat$ to deduce Proposition~\ref{p:exceptional ESC} from the Main Lemma.

\begin{proof}[Proof of Propositions~\ref{p:exceptional ESC} and~\ref{p:real versus complex} assuming the Main Lemma]
Let~$a$ and~$\tau$ be given by the Main Lemma.
Note that we have~$|D\fat| \ge 2a > 1$ on~$\R \setminus [-1, 1]$, so for each~$x < - 1$ we have
$$ \fat(-x) = \fat(x) < x. $$
Let~$A > 3$ be sufficiently large so that~$\fat([-2, 2])$ is contained in~$(- A, A)$, put~$I \= [-A, A]$, and let~$f_0 : I \to I$ be defined by~$f_0(x) \= |x| /2  - 3 A / 2$.
Moreover, let~$f : I \to I$ be a smooth function that coincides with~$\fat$ on~$[- 2, 2]$, that coincides with~$f_0$ on~$[- A, - 3] \cup [3, A]$, and such that for each~$x$ in~$(- A, -1)$ we have
$$ Df(x) > 0
\text{ and }
f(-x) = f(x) < x. $$
Note in particular that~$f$ is a non-degenerate smooth map, that~$x = - A$ is a hyperbolic attracting fixed point of~$f$, and that every point of~$I \setminus [-1, 1]$ converges to~$x = - A$ under forward iteration.
Since~$f$ coincides with~$\fat$ on~$[-1, 1]$ and since~$\fat$ is topologically exact on~$\Kat$, it follows that the Julia set of~$f$ is equal to~$\Kat$.
So by~\cite[Main Theorem']{Riv1204} we have~$\chiinf(f) = \chiper(\fat)$ and by the Main Lemma~$f$ satisfies the conclusions of Proposition~\ref{p:exceptional ESC} with~$p = - 1$.

To prove Proposition~\ref{p:real versus complex} with~$P = \fat$, we show that~$\fat$ satisfies the Collet-Eckmann condition as a complex polynomial and that every invariant probability measure of~$\fat$ on~$\C$ is supported on the Julia set of~$\fat$.
This implies that~$\chiinf^{\C}(\fat) > 0$ and that for every sufficiently small~$\delta > 0$ the~$\liminf$ in the Main Lemma is less than or equal to~$- \chiinf^{\C}(\fat)$, see~\cite[Main Theorem]{PrzRivSmi03}.
So by the Main Lemma we have~$\chiinf^{\C}(\fat) < \chiper(\fat)$ and then the inequality~$\chiinf^{\C}(\fat) < \chiinf^{\R}(\fat)$ follows from the fact that~$\chiper(\fat)$ is equal to~$\chiinff = \chiinf^{\R}(\fat)$, see~\cite[Main Theorem']{Riv1204}.
To prove that~$\fat$ satisfies the Collet-Eckmann condition as a complex polynomial, note that the fact that all the critical points of~$\fat$ in~$\C$ are real and that~$x = 0$ is the only critical point of~$\fat$ in~$\Kat$, implies that each critical point of the complex polynomial~$\fat$ different from~$x = 0$ escapes to~$\infty$ in the Riemann sphere.
So all the critical points of~$\fat$ different from~$x = 0$ are in the Fatou set of~$\fat$.
Together with~\eqref{e:CE condition}, this implies that~$\fat$ satisfies the Collet-Eckmann condition as a complex polynomial.
On the other hand, it follows that the Fatou set of~$\fat$ coincides with the attracting basin of~$\infty$ of~$\fat$, see~\cite[Theorem~$1$]{GraSmi98}.
So every invariant probability measure of~$\fat$ on~$\C$ is supported on the Julia set of~$\fat$.
This completes the proof of Proposition~\ref{p:real versus complex}.
\end{proof}

\section{Mapping properties}
\label{s:mapping properties}
The purpose of this section is to prove some mapping properties of the maps introduced in the previous section.

Note that for~$a > 0$ and~$\tau \le 2$ we have~$\fat(\R \setminus [-1, 1]) = (-\infty, -1)$.
We also have~$|D\fat| \ge 2a$ on~$(- \infty, -1]$; so, if in addition~$a > 1/2$, then every point in~$\R \setminus [-1, 1]$ converges to~$- \infty$ under forward iteration.

Note on the other hand that for every~$a$ and~$\tau$ the point~$x = 0$ is a critical point of~$\fat$ and for~$\tau$ in~$[0, 2]$ the critical value~$\fat(0) = 1 - \tau$ of~$\fat$ is in~$[-1, 1]$.
If in addition~$a > 0$, then~$\fat$ has~$2$ distinct critical points different from $x = 0$: a critical point~$c_-$ in~$(-1, 0)$ and a critical point~$c_+$ in~$(0, 1)$.
Since~$\fat$ is a polynomial of degree~$4$, it follows that all of the critical points of~$\fat$ are real and hence that~$\fat$ has negative Schwarzian derivative on~$\R$.
By symmetry,
$$ c_- = - c_+
\text{ and }
\vat \= \fat(c_+) = \fat(c_-). $$
On the other hand, for~$a > 0$ and~$\tau$ in~$[0, 2]$ we have
\begin{equation}
\label{e:critical value}
\vat
=
1 - \tau + \frac{a^2}{4 (a + 2 - \tau)}
\ge
\frac{a^2}{4 (a + 2)} - 1.
\end{equation}
So, for~$a \ge 10$ we have~$\vat > 1$ and therefore neither~$c_-$ or~$c_+$ is in~$\Kat$.
Since~$\Kat$ contains~$x = 1$ and~$x = - 1$, it follows that~$\Kat$ is disconnected.
Note that~$\vat > 1$ also implies that $\fat^{-1}([-1, 1])$ consists of~$3$ connected components; denote by~$I_{0, a, \tau}$ (resp.~$\Vat$, $I_{1, a, \tau}$) the connected component of this set containing~$x = -1$ (resp.~$0$, $1$).
Note in particular that~$\Kat$ is contained in~$I_{0, a, \tau} \cup \Vat \cup I_{1, a, \tau}$.
\begin{lemm}
\label{l:macro}
For every~$\eta > 1$ there is~$a_0 \ge 20$ such that for every~$a \ge a_0$ and every~$\tau$ in~$[0, 2]$ the following properties hold.
\begin{enumerate}[1.]
\item 
Suppose the forward orbit of~$x = 0$ under~$\fat$ is contained in~$[-1, 1]$.
Then for every integer~$m \ge 1$ and every interval~$J$ such that~$\fat^m$ maps~$J$ diffeomorphically to~$[-1, 1]$, the map~$\fat^m$ maps a neighborhood~$\hJ$ of~$J$ diffeomorphically to~$[-2, 2]$ with distortion bounded by~$\eta$.
\item 
The interval~$I_{0,a,\tau}$ (resp.~$\Vat$, $I_{1, a, \tau}$) is contained in
$$ [-1, -1 + 2/a]
\left( \text{resp. } \left[-2 \sqrt{\tau/a}, 2 \sqrt{\tau/a} \right],
[1 - 2/a, 1] \right). $$
\item
If we put
$$ \lambdaat \= D\fat(-1) = 2(a + 4 - 2 \tau), $$
then for every~$x$ in~$\Vat$ such that~$\fat(x)$ is in~$I_{1, a, \tau}$, we have
$$ \eta^{-1}
\le
\frac{|x|}{(\sqrt{2} \lambdaat^{- 1}) |\fat^2(0) - \fat^2(x)|^{1/2}}
\le
\eta $$
and
$$ \eta^{-1}
\le
\frac{ |D\fat^2(x)|}{ (\sqrt{2} \lambdaat) |\fat^2(0) - \fat^2(x)|^{1/2}}
\le
\eta. $$
\end{enumerate}
\end{lemm}
In the proof of this lemma we use the Koebe principle for maps with negative Schwarzian derivative, see for example~\cite[\S{IV}, $1$]{dMevSt93}.
\begin{proof}
Recall that for~$a \ge 10$ and~$\tau$ in~$[0, 2]$ we have~$\vat > 1$.
This implies that~$\fat(\vat) < - \vat$ and that the forward orbit of~$\vat$ is disjoint from~$(\fat(\vat), \vat)$.
So, if the orbit~$x = 0$ is contained in~$[-1, 1]$ and if~$J$ is an interval and~$m \ge 1$ an integer such that~$\fat^m$ maps~$J$ diffeomorphically to~$[-1, 1]$, then~$\fat^m$ maps a neighborhood of~$J$ diffeomorphically to~$(\fat(\vat), \vat)$.
In view of the Koebe principle, to prove part~$1$ it is enough to show that for every~$A > 1$ there is~$a_0 \ge 10$ such that for every~$a \ge a_0$ and every~$\tau$ in~$[0, 2]$ we have~$\vat \ge A$ and~$\fat(\vat) \le - A$.
To prove the first inequality, note that for~$a \ge 8$ and~$\tau$ in~$[0, 2]$ we have by~\eqref{e:critical value} that~$\vat \ge a/5 - 1$.
So for~$a \ge \max \{ 10, 5(A + 1) \}$ we have
\begin{equation*}
- f(\vat) > \vat \ge A.  
\end{equation*}
This completes the proof of part~$1$.

To prove part~$2$, note that for~$a \ge 10$ and~$\tau$ in~$[0, 2]$ there are~$4$ solutions to the equation~$\fat(x) = 1$, counted with multiplicity.
One the other hand, we have
$$ \fat \left( 2 \sqrt{\tau / a} \right)
=
1 + 3\tau - 16 \frac{(a + 2 - \tau) \tau^2}{a^2}
\ge
1 + \tau \left( 3 - 32 \frac{a + 2}{a^2} \right). $$
Thus for~$a \ge 20$ we have
$$ \fat \left(- 2 \sqrt{\tau / a} \right)
=
\fat \left( 2 \sqrt{\tau / a} \right)
>
1. $$
Since~$\fat(0) \le 1$, it follows that~$\Vat$ is contained in~$\left[- 2 \sqrt{\tau / a}, 2 \sqrt{\tau / a} \right]$.
On the other hand, for~$a \ge 10$ and~$\tau$ in~$[0, 2]$ we have
$$ \fat \left( - 1 + \frac{2}{a} \right)
=
\fat \left( 1 - \frac{2}{a} \right)
\ge
f_{a,2} \left( 1 - \frac{2}{a} \right)
=
4 \frac{(a - 2)^2 (a - 1)}{a^3} - 1
>
1. $$
This proves that~$I_{0, a, \tau}$ is contained in~$[-1, -1 + 2/a]$ and that~$I_{1, a, \tau}$ is contained in~$[1, 1 + 2/a]$.

To prove part~$3$, let~$a_0 \ge 20$ be sufficiently large so part~$2$ holds.
Then for every~$a \ge a_0$ and every~$x$ in~$\Vat$ we have
$$ \frac{|\fat(0) - \fat(x)|}{(\lambdaat/2) x^2}
=
\frac{a}{a + 4 - 2 \tau} \left( 1 - \frac{a + 2 - \tau}{a} x^2 \right)
\in
\left[ \frac{a}{a + 4} \left(1 - 8 \frac{a + 2}{a^2} \right), 1 \right] $$
and
$$ \frac{D\fat(x)}{\lambdaat x}
=
\frac{a}{a + 4 - 2 \tau} \left( 1 - 2 \frac{a + 2 - \tau}{a} x^2 \right)
\in
\left[ \frac{a}{a + 4} \left(1 - 16 \frac{a + 2}{a^2} \right), 1 \right]. $$
This proves that there is a constant~$a_1 \ge a_0$ such that if~$a \ge a_1$ and~$\tau$ is in~$[0, 2]$, then for every~$x$ in~$\Vat$ we have
$$ \eta^{-1/2}
\le
\frac{|x|}{\sqrt{2 / \lambdaat^{-1}} |\fat(0) - \fat(x)|^{1/2}}
\le
\eta^{1/2} $$
and
$$ \eta^{-1/2}
\le
\frac{| D\fat(x)|}{\sqrt{2 \lambdaat} |\fat(0) - \fat(x)|^{1/2}}
\le
\eta^{1/2}. $$
So, to prove part~$3$ it is enough to show that there is~$a_2 \ge a_1$ such that for every~$a \ge a_2$ and every~$\tau$ in~$[0, 2]$ we have for every~$x'$ in~$I_{1, a, \tau}$,
\begin{equation}
\label{e:bounded distortion right branch}
\eta \lambdaat
\le
|D\fat(x')|
\le
\eta \lambdaat.
\end{equation}
To prove this, just note that for~$a \ge 20$, $\tau$ in~$[0, 2]$, and~$x$ in~$I_{1, a, \tau}$, the number
$$ \frac{D\fat(x)}{ - \lambdaat}
=
\frac{a}{a + 4 - 2 \tau} x \left( 2 \frac{(a + 2 - \tau)}{a} x^2 - 1 \right) $$
is in the interval
$$ \left[ \frac{a}{a + 4} \left(1 - \frac{2}{a} \right) \left( 2 \left(1 - \frac{2}{a} \right)^2 - 1 \right), 1 + \frac{4}{a} \right]. $$
This implies~\eqref{e:bounded distortion right branch} and completes the proof of the lemma.
\end{proof}

\section{Combinatorial type}
\label{s:combinatorial type}
In this section we show there parameters~$a$ and~$\tau$ for which the corresponding map~$\fat$ has some special combinatorial properties (Lemma~\ref{l:realization}).
We also prove some general facts about the dynamics of these maps (Lemma~\ref{l:general properties of combinatorial type}).

Let~$\underline{M} \= \left( M_n \right)_{n = 0}^{+ \infty}$ be a sequence of integers such that~$M_0 = 2$ and such that for every~$n \ge 0$ we have~$M_{n + 1} \ge 2 M_n + 1$.
Given~$a \ge 20$ and~$\tau$ in~$[0, 2]$, we say that \emph{the combinatorics of~$\fat$ is of type~$\underline{M}$} if there is an increasing sequence of points~$(x_n)_{n = 0}^{+ \infty}$ in~$(-1, 0)$, such that the following properties hold for every~$n \ge 0$.

\begin{enumerate}[A.]
\item 
The interval~$V_n \= [x_n, - x_n]$ satisfies~$\fat(V_n) \subset I_{1, a, \tau}$ and there is an interval~$J_n$ containing~$\fat^2(V_n)$, such that~$\fat^{M_n - 2}$ maps~$J_n$ diffeomorphically to~$[-1, 1]$, preserving the orientation.
Moreover
$$ \fat^{M_n}(x_n) = -1,
\fat^{M_n}(x_{n + 1}) = x_n,
\text{ and }
\fat^{M_n}(0) \in  V_n \cap (-1, 0). $$
\item 
For every~$j$ in~$\{0, \ldots, M_{n + 1} - 2 M_n - 1 \}$ the point~$\fat^{2M_n + j}(0)$ is in~$I_{0, a, \tau}$.
\end{enumerate}

Note that if the combinatorics of~$\fat$ is of type~$\underline{M}$, then~$V_0 = \Vat$ and for every~$n \ge 0$ the interval~$V_{n  + 1}$ is the pull-back of~$V_n$ by~$\fat^{M_n}$ containing~$x = 0$.
\begin{lemm}
\label{l:realization}
Let~$\underline{M} \= \left( M_n \right)_{n = 0}^{+ \infty}$ be a sequence of integers such that~$M_0 = 2$ and such that for every~$n \ge 0$ we have~$M_{n + 1} \ge 2 M_n + 1$.
Then for every~$a \ge 20$ there is~$\tau$ in~$[0, 2]$ such that the combinatorics of~$\fat$ is of type~$\underline{M}$.
\end{lemm}
\begin{proof}
We define recursively a nested sequence of closed intervals~$(\cT_n)_{n = 0}^{+ \infty}$ and for every~$n \ge 0$ and~$\tau$ in~$\cT_n$ a point~$x_{n, \tau}$, such that the following properties hold for each~$n \ge 0$:
\begin{itemize}
\item
The point~$x_{n, \tau}$ depends continuously with~$\tau$ and~$\fat^{M_n}(x_{n, \tau}) = -1$.
\item
The point~$\fat^{M_n}(0)$ is equal to~$- 1$ if~$\tau$ is the left end point of~$\cT_n$ and it is equal to~$1$ if~$\tau$ is the right end point of~$\cT_n$.
\item
The set~$\fat^{M_n}([x_{n, \tau}, - x_{n, \tau}])$ is contained in~$[-1, 1]$ and the pull-back of~$[-1, 1]$ by~$\fat^{M_n - 2}$ containing~$\fat^2(x_{n, \tau})$ is diffeomorphic and~$\fat^{M_n - 2}$ preserves the orientation on this set.
\end{itemize}

For each~$\tau$ in~$[0, 2]$ let~$x_{0, \tau}$ be the left end point of~$\Vat$.
Clearly~$x_{0, \tau}$ depends continuously with~$\tau$ on~$[0, 2]$ and~$\fat^2(x_{0, \tau}) = -1$.
To define~$\cT_0$, note that by the intermediate value theorem there is~$\tau$ in~$[0, 2]$ such that~$\fat(0)$ is equal to the left end point of~$I_{1, a, \tau}$.
Let~$\tau_0$ be the least number with this property and put~$\cT_{0} \= [0, \tau_0]$.
Then for every~$\tau$ in~$\cT_0$ the set~$\fat(\Vat)$ is contained in~$I_{1, a, \tau}$ and therefore~$\fat^2(\Vat)$ is contained in~$[-1, 1]$.
Moreover, note that~$f^2_{a, 0}(0) =  - 1$ and~$f_{a, \tau_0}^2(0) = 1$, so all of the properties above are satisfied when~$n = 0$.

Let~$n \ge 0$ be an integer and suppose by induction that a closed interval~$\cT_n$ and for every~$\tau$ in~$\cT_n$ a point~$x_{n, \tau}$ are already defined.
By the intermediate value theorem we can find parameters~$\tau_{n + 1}^{-}$ and~$\tau_{n + 1}^{+}$ in~$\cT_n$ such that,
$$ \tau_{n + 1}^{-} < \tau_{n + 1}^{+},
f_{a, \tau_{n + 1}^{-}}^{M_n}(0) = x_{n, \tau},
f_{a, \tau_{n + 1}^{+}}^{M_n}(0) = 0, $$
and such that for every~$\tau$ in~$\cT_{n + 1}' \= [\tau_{n + 1}^{-}, \tau_{n + 1}^{+}]$ we have
$$ x_{n, \tau} \le \fat^{M_n}(0) \le 0. $$
For each~$\tau$ in~$\cT_{n + 1}'$, let~$x_{n + 1, \tau}$ be the unique point in~$[-1, 0]$ that is mapped to~$x_{n, \tau}$ by~$\fat^{M_n}$.
Clearly, $x_{n + 1, \tau}$ depends continuously with~$\tau$ on~$\cT_{n + 1}'$.
Moreover,
$$ \fat^{M_{n + 1}}(x_{n + 1, \tau})
=
\fat^{M_{n + 1} - M_n}(x_{n, \tau})
=
\fat^{M_{n + 1} - 2M_n} (-1)
=
-1. $$
So for each~$\tau$ in~$\cT_{n + 1}'$ the first property above is satisfied with~$n$ replaced by~$n + 1$.
By the induction hypothesis for every~$\tau$ in~$\cT_{n + 1}'$ we have
$$ \fat^{2M_n}([x_{n + 1, \tau}, - x_{n + 1, \tau}])
\subset
\fat^{M_n}([x_{n, \tau}, - x_{n, \tau}])
\subset
[-1, 1]. $$
Moreover, the pull-back of~$I_{0, a, \tau}$ by~$f^{2M_n - 2}$ containing~$\fat^2(x_{n + 1, \tau})$ is diffeomorphic and~$\fat^{2M_n - 2}$ preserves the orientation on this set.
Since~$M_{n + 1} \ge 2 M_n  + 1$, it follows that there is a parameter~$\tau$ in~$\cT_{n + 1}'$ such that~$\fat^{M_{n + 1}}(0) = 1$.
Let~$\tau_{n + 1}$ be the least number with this property and put~$\cT_{n + 1} \= [\tau_{n + 1}^-, \tau_{n + 1}]$.
Since we also have,
$$ f_{a, \tau_n^-}^{M_{n + 1}}(0)
=
f_{a, \tau_n^-}^{M_{n + 1} - M_n}(x_{n, \tau})
=
f_{a, \tau_n^-}^{M_{n + 1} - 2M_n}(-1)
= -1, $$
the second point is satisfied with~$n$ replaced by~$n + 1$.
To prove that for each~$\tau$ in~$\cT_{n + 1}$ the third point is satisfied with~$n$ replaced by~$n + 1$, note that by the definition of~$\tau_{n + 1}$ for each~$j$ in~$\{0, 1, \ldots, M_{n + 1} - 2 M_n - 1 \}$ we have
\begin{equation}
  \label{e:nth itinerary}
  \fat^{2M_n + j}([x_{n + 1, \tau}, - x_{n + 1, \tau}]) \subset I_{0, a, \tau}.
\end{equation}
Moreover,
$$ \fat^{M_{n + 1}}([x_{n + 1, \tau}, - x_{n + 1, \tau}])
\subset [-1, 1]. $$
It follows that the pull-back of~$[-1, 1]$ by~$\fat^{M_{n + 1} - 2 M_n}$ containing~$x = -1$ is diffeomorphic and contained in~$I_{0, a, \tau}$ and that~$\fat^{M_{n + 1} - 2 M_n}$ preserves the orientation on this set.
By the considerations above this implies that the pull-back of~$[-1, 1]$ by~$\fat^{M_{n + 1} - 2}$ containing~$\fat^2(0)$ is diffeomorphic and that~$\fat^{M_{n + 1} - 2}$ preserves the orientation on this set.
This proves that the third point with~$n$ replaced by~$n + 1$.

To prove the lemma, let~$\tau$ be a parameter in the intersection~$\bigcap_{n = 0}^{+ \infty} \cT_n$ and for each integer~$n \ge 0$ put~$x_n \= x_{n, \tau}$.
Property~A follows directly from the induction hypothesis and the definitions above.
Property~B is given by~\eqref{e:nth itinerary}.
The proof of the lemma is thus complete.
\end{proof}

\begin{lemm}
\label{l:general properties of combinatorial type}
Let~$\underline{M} \= \left( M_n \right)_{n = 0}^{+ \infty}$ be a sequence of integers such that~$M_0 = 2$ and such that for every~$n \ge 0$ we have~$M_{n + 1} \ge 2 M_n + 1$.
Suppose moreover that
\begin{equation}
  \label{e:unbounded combinatorics}
M_{n + 1} - 2M_n \to + \infty
\text{ as }
n \to + \infty  
\end{equation}
and let~$a \ge 20$ and~$\tau$ in~$[0, 2]$ be such that the combinatorics of~$\fat$ is of type~$\underline{M}$.
Then the forward orbit of~$x = 0$ is contained in~$[-1, 1]$ and~$\fat$ is topologically exact on~$\Kat$.
\end{lemm}
\begin{proof}
To prove the first assertion note that for every~$n \ge 0$ the point~$\fat^{M_n}(0)$ is in~$[-1, 1]$.
Since~$M_n \to + \infty$ as~$n \to + \infty$, it follows that for every integer~$m \ge 0$ the point~$\fat^m(0)$ is in~$[-1, 1]$.

To prove that~$\fat$ is topologically exact on~$\Kat$, we first make some preparatory remarks.
Property~B and~\eqref{e:unbounded combinatorics} imply that~$x = 0$ cannot be asymptotic to a periodic cycle.
Using that~$\fat$ has negative Schwarzian derivative, Singer's theorem implies that all the periodic points of~$\fat$ are hyperbolic repelling, see for example~\cite[Theorem~II.$6$.$1$]{dMevSt93}.
Since~$\fat$ has no wandering intervals intersecting~$\Kat$, see for example~\cite[Theorem~A in~{\S}IV]{dMevSt93}, it follows that the interior of~$\Kat$ is empty.
In particular, $x_n \to 0$ as~$n \to + \infty$.

We proceed to the proof that~$\fat$ is topologically exact on~$\Kat$.
To do this, let~$U$ be an open interval intersecting~$\Kat$.
If for some integer~$m \ge 1$ the interval~$\fat^m(U)$ contains~$x = 0$, then there is an integer~$n \ge 0$ such that~$\fat^m(U)$ contains either~$[x_n, 0]$ or~$[0, - x_n]$.
In both cases~$I_{0, a, \tau} \subset \fat^{M_n + m}(U)$ and therefore~$[-1, 1] \subset \fat^{M_n + m + 1}(U)$.
It remains to consider the case where for each integer~$m \ge 0$ the interval~$\fat^m(U)$ does not contain~$x = 0$.
Since~$\Kat$ has empty interior, it follows that there is an integer~$m \ge 0$ such that~$\fat^m(U)$ is not contained in~$I_{0, a, \tau} \cup \Vat \cup I_{1, a, \tau}$.
If~$m$ is the least integer with this property, then~$\fat^m$ maps~$U$ diffeomorphically to~$\fat^m(U)$.
Thus~$\fat^m(U)$ contains one of the points~$\partial I_{0, a, \tau} \cup \partial \Vat \cup \partial I_{1, a, \tau}$ in its interior and therefore~$\fat^{m + 2}(U)$ contains a neighborhood of the hyperbolic repelling fixed point~$x = -1$.
Again using that~$\Kat$ has empty interior we conclude that there is an integer~$N \ge 0$ such that~$\fat^{m + N + 2}(U)$ contains~$I_{0, a, \tau}$ and therefore~$\fat^{m + N + 3}(U) \supset [-1, 1]$.
This completes the proof that~$\fat$ is topologically exact on~$\Kat$ and of the lemma.
\end{proof}

\section{Main estimates}
\label{s:estimates}
We start this section chosing the map used to prove the Main Lemma, through its combinatorial type.
The rest of this section is devoted to prove some estimates concerning this map (Lemmas~\ref{l:close return expansion} and~\ref{l:long branch}), that are used in the proof of the Main Lemma in the next section.

Fix~$\eta$ in~$(1, 2)$, let~$a_0 \ge 20$ be given by Lemma~\ref{l:macro} for this choice of~$\eta$, and fix~$a$ satisfying
\begin{equation}
  \label{e:main parameter lower bound}
  a > \max \{a_0, 128, \eta^{20} \}.
\end{equation}
Define a sequence of integers~$\underline{M} \= (M_n)_{n = 0}^{+ \infty}$ inductively by~$M_0 = 2$ and by the property that for every~$n \ge 0$ we have
\begin{equation}
\label{e:return times growth 0}
\eta^{M_{n + 1}}
\ge
4 \eta^{5M_n/2} (2a + 8)^{M_n / 2}.
\end{equation}
Note that for each~$n \ge 0$ we have
\begin{equation}
\label{e:return times grow exponentially}
M_{n + 1}
\ge
5M_n/2
\end{equation}
and~$M_{n + 1} - 2M_n \ge M_n/2 \ge 2^n$.
In particular,
$$ M_{n + 1} - 2M_n \to + \infty
\text{ as }
n \to + \infty. $$

Let~$\tau$ in~$[0, 2]$ be such that the combinatorics of~$\fat$ is of type~$\underline{M}$ (Lemma~\ref{l:realization}).
In what follows we put
$$ f \= \fat,
K \= \Kat,
I_0 \= I_{0, a, \tau},
\emph{etc}. $$
Note that, by~\eqref{e:main parameter lower bound} and by the definition of~$\lambda$, we have
\begin{equation}
  \label{e:subexponential error}
  \lambda > \max \{256, \eta^{20} \}.
\end{equation}
Moreover, for every integer~$n \ge 0$ we have by~\eqref{e:return times growth 0},
\begin{equation}
\label{e:return times growth}
\eta^{M_{n + 1}}
\ge
4 \eta^{5M_n/2} \lambda^{M_n / 2}.
\end{equation}
\begin{lemm}
\label{l:close return expansion}
For every integer~$n \ge 0$ the following properties hold.
\begin{enumerate}[1.]
\item 
On~$J_n$ we have
$$ \eta^{- (M_n - 2)} \lambda^{M_n - 2}
\le
|Df^{M_n - 2}|
\le
\eta^{M_n - 2} \lambda^{M_n - 2}. $$
\item
We have
\begin{equation}
  \label{e:cutting point}
\eta^{- M_n/2} \lambda^{- M_n/2}
\le
|x_n|
\le
\sqrt{2} \eta^{M_n/2} \lambda^{- M_n/2}
\end{equation}
and
\begin{equation}
\label{e:close return derivative}
\eta^{- (3 M_n/2 - 2)} \lambda^{M_n/2}
\le
|Df^{M_n}(x_n)|
\le
\sqrt{2} \eta^{3 M_n/2 - 2} \lambda^{M_n/2}.
\end{equation}
\end{enumerate}
\end{lemm}
\begin{proof}
Since~$M_0 = 2$, part~$1$ holds trivially when~$n = 0$.
Suppose by induction that part~$1$ holds for some integer~$n \ge 0$.
Property~A implies~$f^{M_n}(V_n) = [-1, f^{M_n}(0)]$.
Since~$f^{M_n}(0)$ is~$V_n \cap (-1, 0)$ (property~A) and since this last set is contained in~$[-1/4, 0]$ (\eqref{e:main parameter lower bound} and part~$2$ of Lemma~\ref{l:macro}), the induction hypothesis implies that
$$ \frac{1}{2} \eta^{-(M_n - 2)} \lambda^{-(M_n - 2)}
\le
|f^2(V_n)|
\le
\eta^{M_n - 2} \lambda^{-(M_n - 2)}. $$
Thus, by part~$3$ of Lemma~\ref{l:macro} we obtain~\eqref{e:cutting point} and
$$ \eta^{- M_n/2} \lambda^{- M_n/2 + 2}
\le
|Df^2(x_n)|
\le
\sqrt{2} \eta^{M_n/2} \lambda^{- M_n/2 + 2}. $$
Using the induction hypothesis again we obtain~\eqref{e:close return derivative}.

To prove part~$1$ with~$n$ replaced by~$n + 1$, note first that by the induction hypothesis we have
$$ \eta^{- (M_n - 2)} \lambda^{M_n - 2}
\le
|Df^{M_n - 2}(f^2(x_{n + 1}))|
\le
\eta^{M_n - 2} \lambda^{M_n - 2}. $$
Since~$f^{M_n}(x_{n + 1}) = x_n$, by~\eqref{e:close return derivative} we have
$$ \eta^{-(5M_n/2 - 4)} \lambda^{3 M_n/2 - 2}
\le
|Df^{2M_n - 2}(f^2(x_{n + 1}))|
\le
\sqrt{2} \eta^{5M_n/2 - 4} \lambda^{3 M_n/2 - 2}. $$
Using that~$f^{2M_n}(x_{n + 1}) = -1$, we obtain
\begin{multline*}
  \eta^{-(5M_n/2 - 4)} \lambda^{M_{n + 1} - M_n/2 - 2}
\le
|Df^{M_{n + 1} - 2}(f^2(x_{n + 1}))|
\\ \le
\sqrt{2} \eta^{5M_n/2 - 4} \lambda^{M_{n + 1} - M_n/2 - 2}. 
\end{multline*}
Using part~$1$ of Lemma~\ref{l:macro} and~\eqref{e:return times growth} we obtain part~$1$ of the lemma with~$n$ replaced by~$n + 1$.
This completes the proof of the induction step and of the lemma.
\end{proof}

Define the sequence of intervals~$(U_n)_{n = 0}^{+ \infty}$ inductively as follows.
Put~$U_0 = [-1, 1] \setminus (I_0 \cup I_1)$ and note that~$U_0$ contains~$V_0$.
Let~$n \ge 0$ be an integer such that~$U_n$ is already defined and contains~$V_n$.
Then, let~$U_{n + 1}$ be the pull-back of~$U_n$ by~$f^{M_n}$ containing~$x = 0$.
By definition~$U_{n + 1}$ contains~$V_{n + 1}$.

Using that the closure of~$U_0$ is contained in~$(-1, 1)$, an induction argument shows that for every~$n \ge 0$ the interval~$V_n$ contains the closure of~$U_{n + 1}$ in its interior.
So, if for each~$n \ge 0$ we denote by~$y_n$ the left end point of~$U_n$, then for every~$n \ge 0$ we have
$$ f^{M_n}(y_{n + 1}) = y_n
\text{ and }
y_n < x_n < y_{n + 1} < 0. $$ 
\begin{lemm}
\label{l:long branch}
For every integer~$n \ge 0$ we have
\begin{equation}
  \label{e:gap point}
|y_{n + 1}|
\ge
|x_{n + 1} - y_{n + 1}|
\ge
\eta^{- M_n} \lambda^{- M_n/2}.
\end{equation}
Moreover, on~$V_n \setminus U_{n + 1}$ we have
\begin{equation}
  \label{e:gap derivative}
  |Df^{M_n}|
\ge
\eta^{-2 M_n} \lambda^{M_n/2}.
\end{equation}
\end{lemm}
\begin{proof}
Note that, since~$I_0$ is contained in~$[-1, -7/8]$ and~$V$ in~$[-1/4, 1/4]$ (\eqref{e:main parameter lower bound} and part~$2$ of Lemma~\ref{l:macro}), we have $|x_0 - y_0| \ge 5/8$.
On the other hand, for every integer~$n \ge 0$ the point~$f^2(y_{n + 1})$ is contained in~$f^2(V_n) \subset J_n$, so by part~$1$ of Lemma~\ref{l:close return expansion} we have
\begin{equation*}
|f^2(0) - f^2(y_{n + 1})|
\ge
|f^2(x_{n + 1}) - f^2(y_{n + 1})|
\ge
|x_n - y_n| \eta^{-(M_n - 2)} \lambda^{-(M_n - 2)}.
\end{equation*}
Combined with part~$3$ of Lemma~\ref{l:macro} this implies
\begin{equation}
\label{e:bridge size}
|y_{n + 1}|
\ge
\sqrt{2} |x_n - y_n|^{1/2} \eta^{- M_n/2} \lambda^{- M_n/2}.
\end{equation}
When~$n = 0$ we obtain~$|y_1| \ge \sqrt{5/4} \eta^{-1} \lambda^{-1}$.
Together with~\eqref{e:subexponential error}, with \eqref{e:return times growth} with~$n = 0$, and with~\eqref{e:cutting point} with~$n = 1$, we have
$$ |x_1|
\le
\sqrt{2} \eta^{M_1/2} \lambda^{- M_1/2}
\le
\sqrt{2} \eta^{- 2 M_1}
\le
\eta^{-2} \lambda^{- 2}
\le
|y_1| - \eta^{- 2} \lambda^{- 1}. $$
This implies~\eqref{e:gap point} when~$n = 0$.
To prove that~\eqref{e:gap point} holds for~$n \ge 1$, we proceed by induction.
Let~$n \ge 1$ be an integer such that~\eqref{e:gap point} holds with~$n$ replaced by~$n - 1$.
Using~\eqref{e:return times growth}, \eqref{e:bridge size}, and the induction hypothesis, we obtain
\begin{equation}
\label{e:gap size n}
|y_{n + 1}|
\ge
\sqrt{2} \eta^{- M_n/2 - M_{n - 1}/2} \lambda^{- M_n/2 - M_{n - 1}/4}
\ge
\sqrt{2} \eta^{- M_n} \lambda^{- M_n/2}.
\end{equation}
Together with~\eqref{e:return times grow exponentially}, \eqref{e:subexponential error}, \eqref{e:return times growth}, and~\eqref{e:cutting point} with~$n$ replaced by~$n + 1$, we have
\begin{multline*}
|x_{n + 1}|
\le
\sqrt{2} \eta^{M_{n + 1}/2} \lambda^{- M_{n + 1}/2}
\le
\sqrt{2} \eta^{- M_{n + 1}}
\le
(\sqrt{2} - 1) \eta^{- M_n} \lambda^{- M_n/2}
\\ \le
|y_{n + 1}| - \eta^{-M_n} \lambda^{- M_n/2}.
\end{multline*}
This completes the proof of the induction step and of the fact that~\eqref{e:gap point} holds for every~$n \ge 0$.

To prove~\eqref{e:gap derivative}, let~$n \ge 0$ be an integer and note that by~\eqref{e:gap point} and by part~$3$ of Lemma~\ref{l:macro}, for each~$x$ in~$[x_n, y_{n + 1}]$ we have
$$ |Df^2(- x)|
=
|Df^2(x)|
\ge
\eta^{-2} \lambda^{2} |x|
\ge
\eta^{-2} \lambda^{2} |y_{n + 1}|
\ge
\eta^{- M_n - 2} \lambda^{- M_n/2 + 2}. $$
Together with part~$1$ of Lemma~\ref{l:close return expansion}, this implies~\eqref{e:gap derivative}.
The proof of the lemma is thus complete.
\end{proof}

\section{Proof of the Main Lemma}
\label{s:proof of the Main Lemma}
After showing some general properties of the dynamics of the map chosen in the previous section (Lemma~\ref{l:showing CE}), in this section we give the proof of the Main Lemma.

Throughout this section we use the notation introduced in the previous section.
Note that part~$1$ of Lemma~\ref{l:macro} implies that on~$I_0 \cup I_1$ we have
\begin{equation}
\label{e:hyperbolic part}
\eta^{-1} \lambda \le |Df| \le \eta \lambda.
\end{equation}
\begin{lemm}
\label{l:showing CE}
The map~$f$ is topologically exact on the maximal invariant set~$K$ of~$f$ in~$[-1, 1]$ and~$f$ satisfies the Collet-Eckmann condition.
\end{lemm}
\begin{proof}
That~$f$ is topologically exact on~$K$ is given by Lemma~\ref{l:general properties of combinatorial type}.
To prove that~$\fat$ satisfies the Collet-Eckmann condition, in part~$1$ we show that for every integer~$n \ge 0$ we have
\begin{equation}
  \label{e:derivative at closest returns}
  |Df^{M_n}(f(0))|
\ge
\eta^{- 2M_n - 2} \lambda^{- M_n / 2}.
\end{equation}
In part~$2$ we complete the proof of the lemma using this fact.

\partn{1}
By part~$1$ of Lemma~\ref{l:close return expansion} and~\eqref{e:hyperbolic part}, we have
\begin{equation}
  \label{e:derivative before close return}
|Df^{M_n - 1}(f(0))|
=
|Df^{M_n - 2}(f^2(0))| \cdot |Df(f(0))|
\ge
(\eta^{-1} \lambda)^{M_n - 1}.
\end{equation}
To estimate~$Df(f^{M_n}(0))$, note that the property that~$f^{2M_n}(0)$ is in~$I_0$, implies~$x_n \le f^{M_n}(0) < x_{n + 1}$.
Since~$(y_{n + 1}, x_{n + 1}) \subset U_{n + 1} \setminus V_{n + 1}$ is disjoint from~$K$, it follows that~$x_n \le f^{M_n}(0) \le y_{n + 1}$.
By~\eqref{e:hyperbolic part}, by part~$3$ of Lemma~\ref{l:macro}, by~\eqref{e:gap point}, and by~\eqref{e:gap derivative}, we have
\begin{multline*}
|Df(f^{M_n}(0))|
=
|Df^2(f^{M_n}(0))| \cdot |Df(f^{M_n  + 1}(0))|^{-1}
\\ \ge
\eta^{-1} \lambda^{-1} |Df^2(f^{M_n}(0))|
\\ \ge
\eta^{-3} \lambda |y_{n + 1}|
\ge
\eta^{- M_n - 3} \lambda^{- M_n/2 + 1}.
\end{multline*}
Combined with~\eqref{e:derivative before close return}, this implies~\eqref{e:derivative at closest returns}.

\partn{2}
To prove that~$f$ satisfies the Collet-Eckmann condition we show by induction that for every integer~$m \ge 1$ we have
\begin{equation}
\label{e:lower Lyapunov}
|Df^m(f(0))| \ge (\eta^{-3} \lambda^{1/2})^m.
\end{equation}
When~$m = 1$ this inequality follows from~\eqref{e:hyperbolic part} and from the fact that~$f(0)$ is in~$I_1$.
Let~$m \ge 2$ be an integer and suppose by induction that~\eqref{e:lower Lyapunov} holds with~$m$ replaced by each element of~$\{1, \ldots, m - 1 \}$.
Let~$n \ge 0$ be the largest integer such that~$M_n \le m$.
When~$m = M_n$ the inequality~\eqref{e:lower Lyapunov} follows from~\eqref{e:derivative at closest returns}.
If~$m = M_n + 1$, then~$f^{M_n + 1}(0)$ is in~$I_1$ and~\eqref{e:lower Lyapunov} follows from~\eqref{e:hyperbolic part} together with the induction hypothesis.
Suppose that~$n \ge 1$ and that~$m$ is in~$\{M_n + 2, \ldots, 2M_n - 1 \}$.
Note that the pull-back~$J_n'$ of~$J_n$ by~$f$ containing~$f(0)$ is contained in~$I_1$ and it is mapped diffeomorphically to~$[-1, 1]$ by~$f^{M_n - 1}$.
Since~$f(0)$ and~$f^{M_n + 1}(0)$ are both in~$f(V_n)$ and hence in~$J_n'$, we have by part~$1$ of Lemma~\ref{l:macro} and by the induction hypothesis
\begin{equation*}
|Df^{m - M_n}(f^{M_n + 1}(0))|
\ge
\eta^{-1} |Df^{m - M_n}(f(0))|
\ge
\eta^{-1} (\eta^{-3} \lambda^{1/2})^{m - M_n}.
\end{equation*}
Combined with~\eqref{e:derivative at closest returns}, this implies~\eqref{e:lower Lyapunov}.
Finally, if~$m$ is in~$\{ 2M_n, \ldots, M_{n + 1} - 1 \}$, then~$f^m(0)$ is in~$I_1$ and~\eqref{e:lower Lyapunov} follows from~\eqref{e:hyperbolic part} and the induction hypothesis.
This completes the proof of the induction step and of the fact that~$f$ satisfies the Collet-Eckmann condition.  
\end{proof}

\begin{proof}[Proof of the Main Lemma]
In part~$1$ below we show
\begin{equation}
\label{e:asymptotic expansion lower bound}
\chiper(f) \ge \frac{1}{2} \ln \lambda - 2 \ln \eta > 0.
\end{equation}
Parts~$2$ and~$3$ are devoted to show the last inequality of the Main Lemma; the remaining statements are given by Lemma~\ref{l:showing CE}.
In part~$2$ we reduce the proof of the desired inequality to a lower estimate on the size of a certain sequence of pull-backs.
The lower estimate is given in part~$3$.

\partn{1}
To prove~\eqref{e:asymptotic expansion lower bound}, let
$$ m : [-1, 1] \setminus \{ 0 \} \to \N $$
be the function that is constant equal to~$1$ on~$[-1, 1] \setminus V_0$ and that for every integer~$n \ge 0$ is constant equal to~$M_n$ on~$V_n \setminus V_{n + 1}$.
We show that for every~$x$ in~$K \setminus \{ 0 \}$, we have
\begin{equation}
\label{e:induced expansion}
|Df^{m(x)}(x)| \ge (\eta^{-2} \lambda^{1/2})^{m(x)}.
\end{equation}
If~$x$ is in~$K \setminus V_0$, then actually~$x$ is in~$K \setminus U_0 \subset I_0 \cup I_1$ and in this case~\eqref{e:induced expansion} is given by~\eqref{e:hyperbolic part}.
On the other hand, if for some integer~$n \ge 0$ the point~$x$ is in~$(V_n \setminus V_{n + 1}) \cap K$, then~$x$ is in~$V_n \setminus U_{n + 1}$ and~\eqref{e:induced expansion} is given by~\eqref{e:gap derivative}.
Thus~\eqref{e:induced expansion} is proved for each~$x$ in~$K \setminus \{0 \}$.
To prove~\eqref{e:asymptotic expansion lower bound}, let~$p$ be a hyperbolic repelling periodic point of~$f$.
Then~$x = 0$ is not in the forward orbit of~$p$.
Thus we can define inductively a sequence of integers~$(m_\ell)_{\ell = 0}^{+ \infty}$, by~$m_0 \= m(p)$ and for~$\ell \ge 1$, by
$$ m_\ell \= m(f^{m_{\ell - 1}}(p)) + m_{\ell - 1}. $$
Then by~\eqref{e:induced expansion} we have
$$ \chi_p(f)
=
\lim_{m \to + \infty} \frac{1}{m} \ln |Df^m(p)|
=
\lim_{\ell \to + \infty} \frac{1}{m_{\ell}} \ln |Df^{m_{\ell}}(p)|
\ge
\frac{1}{2} \ln \lambda - 2 \ln \eta. $$
This proves~\eqref{e:asymptotic expansion lower bound}.

\partn{2}
Let~$\delta > 0$ be given, let~$N_0 \ge 0$ be an integer such that~$\lambda^{- N_0} \le \delta$, and put
$$ J \= (- 1 - \lambda^{- N_0}, - 1]. $$
Note that for every integer~$n \ge 0$ we have~$f^{M_{n + 1} + 2M_n}(x_{n + 2}) = - 1$.
For each integer~$n \ge 1$ satisfying~$M_{n + 1} - 2M_n \ge N_0$, we show in part~$3$ below that the pull-back~$W_n$ of~$J$ by~$f^{2M_{n + 1} + 2 - N_0}$ containing~$x_{n + 2}$ satisfies,
\begin{equation}
\label{e:after 2 square roots}
|W_n|
\ge
(4 \eta \lambda)^{-1} (\eta \lambda)^{- 3M_{n + 1}/4}.
\end{equation}
In view of~\eqref{e:subexponential error} and~\eqref{e:asymptotic expansion lower bound}, this implies
\begin{multline*}
\liminf_{m \to + \infty} - \frac{1}{m} \ln \{ |W| : W \text{ connected component of } f^{-m}((- 1 - \delta, -1]) \}
\\ \le
\frac{3}{8} \ln (\eta \lambda)
<
\frac{1}{2} \ln \lambda - 2 \ln \eta
\le
\chiper(f).
\end{multline*}
So to complete the proof of the proposition it is enough to show that for every  integer~$n \ge 0$ such that~$M_{n + 1} - 2M_n \ge N_0$, we have~\eqref{e:after 2 square roots}.

\partn{3}
Let~$n \ge 1$ be an integer such that~$M_{n + 1} - 2 M_n \ge N_0$.
It follows from property~A that the pull-back of~$I_0$ by~$f^{M_n}$ containing~$x_n$ is diffeomorphic.
This implies that the pull-back of~$I_0$ by~$f^{2M_n - 2}$ containing~$f^2(x_{n + 1})$ is diffeomorphic and hence that for every integer~$M \ge 2M_n - 1$ the pull-back of~$I$ by~$f^{M}$ containing~$f^2(x_{n + 1})$ is diffeomorphic.
So by part~$1$ of Lemma~\ref{l:macro} the pull-back~$J'$ of~$J$ by~$f^{M_{n + 1} - N_0}$ containing~$f^2(x_{n + 1})$ satisfies
\begin{multline*}
\eta^{-1} \lambda^{-N_0} |Df^{M_{n + 1} - N_0}(f^2(x_{n + 1}))|^{-1}
\le
|J'|
\\ \le
\eta \lambda^{-N_0} |Df^{M_{n + 1} - N_0}(f^2(x_{n + 1}))|^{-1}.
\end{multline*}
So, by part~$1$ of Lemma~\ref{l:close return expansion} and~\eqref{e:close return derivative}, we have
\begin{equation}
\label{e:before first square root}
\sqrt{1 / 2} \eta^{- (5M_n/2 - 3)} \lambda^{- M_{n + 1} + M_n/2}
\le
|J'|
\le
\eta^{5M_n/2 - 3} \lambda^{- M_{n + 1} + M_n/2}.
\end{equation}
On the other hand, by~\eqref{e:cutting point} with~$n$ replaced by~$n + 1$ and by part~$3$ of Lemma~\ref{l:macro} applied to~$x = x_{n + 1}$, we have
\begin{equation}
\label{e:critical value distance}
|f^2(0) - f^2(x_{n + 1})|
\le
\eta^{M_{n + 1} + 2} \lambda^{- M_{n + 1} + 2}.
\end{equation}
Together with~\eqref{e:return times growth} and~\eqref{e:before first square root} this implies that, if we denote by~$z'$ the left end point of~$J'$, then
$$ |f^2(0) - z'|
\le
2 \eta^{M_{n + 1} + 2} \lambda^{- M_{n + 1} + 2}. $$
So by part~$3$ of Lemma~\ref{l:macro} we have
\begin{equation}
\label{e:derivative at close return}
|Df^2|
\le
2 \eta^{M_{n + 1}/2 + 2} \lambda^{- M_{n + 1}/2 + 2}
\end{equation}
on the pull-back~$J''$ of~$J'$ by~$f^2$ containing~$x_{n + 1}$.
Note that~$J''$ is the pull-back of~$J$ by~$f^{M_{n + 1} + 2 - N_0}$ containing~$x_{n + 1} = f^{M_{n + 1}}(x_{n + 2})$.
Combining~\eqref{e:before first square root} and~\eqref{e:derivative at close return}, we obtain
\begin{equation*}
\sqrt{1 / 2} \eta^{- (5M_n/2 - 3)} \lambda^{- M_{n + 1} + M_n/2}
\le
|J'|
\le
|J''| 2 \eta^{M_{n + 1}/2 + 2} \lambda^{- M_{n + 1}/2 + 2}
\end{equation*}
Together with~\eqref{e:return times grow exponentially} this implies
\begin{equation*}
|J''|
\ge
2^{-2} \eta^{- M_{n + 1} / 2} \lambda^{- M_{n + 1}/2 - 2}.  
\end{equation*}
By part~$1$ of Lemma~\ref{l:close return expansion} with~$n$ replaced by~$n + 1$, this implies that the pull-back~$J'''$ of~$J''$ by~$f^{M_{n + 1} - 2}$ containing~$f^2(x_{n + 2})$ satisfies
\begin{equation}
\label{e:before second square root}
|J'''|
\ge
2^{-2} \eta^{- 3 M_{n + 1}/2} \lambda^{- 3M_{n + 1}/2}.
\end{equation}
Combined with~\eqref{e:subexponential error}, with~\eqref{e:return times growth} with~$n$ replaced by~$n + 1$, and with~\eqref{e:critical value distance} with~$n + 1$ replaced by~$n + 2$, we obtain
$$ |f^2(0) - f^2(x_{n + 2})|
\le
\eta^{M_{n + 2} + 2} \lambda^{- M_{n + 2} + 2}
\le
\eta^{- 3 M_{n + 2}}
\le
|J'''|. $$
Noting that~$W_n$ is the pull-back of~$J'''$ by~$f^2$ containing~$x_{n + 2}$, if we denote by~$w_n$ the left end point of~$W_n$, then we have
$$ f^2(w_n) = z'''
\text{ and }
|f^2(0) - f^2(w_n)| \le 2 |J'''|. $$
So by part~$3$ of Lemma~\ref{l:macro} we have
$$ |D f^2| \le 2 \eta \lambda |J'''|^{1/2} $$
on~$W_n$ and therefore
\begin{equation}
|J'''|
\le
2 \eta \lambda |J'''|^{1/2} \cdot |W_n|.
\end{equation}
Together with~\eqref{e:before second square root} this implies
$$ |W_n|
\ge
(2 \eta \lambda)^{-1} |J'''|^{1/2}
\ge
(4 \eta \lambda)^{-1} \eta^{- 3 M_{n + 1}/4} \lambda^{- 3 M_{n + 1} /4}. $$
This proves~\eqref{e:after 2 square roots} and completes the proof of the proposition.
\end{proof}

\section{Proof of Theorem~\ref{t:qualitative asymptotic expansion}}
\label{s:proof of ESC revisited}
The purpose of this section is to prove Theorem~\ref{t:qualitative asymptotic expansion}.
First we use~\cite[Main Theorem']{Riv1204} to reduce the proof of Theorem~\ref{t:qualitative asymptotic expansion} to show the exponential shrinking of components of a small neighborhood of a periodic point in the boundary of a Fatou component (Proposition~\ref{p:ESC revisited}).
For future reference we state this last result for a more general class of maps.
The proof of this result occupies most of this section.

A non-injective interval map~$f : I \to I$ is \emph{of class~$C^3$ with non-flat critical points} if:
\begin{itemize}
\item
The map~$f$ is of class~$C^3$ outside $\Crit(f)$.
\item
For each critical point~$c$ of~$f$ there exists a number $\ell_c>1$ and diffeomorphisms~$\phi$ and~$\psi$ of~$\R$ of class~$C^3$, such that $\phi(c)=\psi(f(c))=0$ and such that on a neighborhood of~$c$ on~$I$ we have,
$$ |\psi\circ f| = \pm |\phi|^{\ell_c}. $$
The number~$\ell_c$ is the \emph{order of~$f$ at~$c$}.
\end{itemize}
Denote by~$\sA$ the collection of interval maps of class~$C^3$ with non-flat critical points, whose Julia set is completely invariant.
Note that each smooth non-degenerate map whose Julia set is completely invariant is contained in~$\sA$.
In Appendix~\ref{s:general properties} we gather some general properties of maps in~$\sA$ that are used below.

By~\cite[\S$1$]{MardMevSt92}, each interval map in~$\sA$ has at most a finite number of periodic Fatou components.
Combined with Lemma~\ref{l:backward Lyapunov} and part~$1$ of the Main Theorem' of~\cite{Riv1204}, Theorem~\ref{t:qualitative asymptotic expansion} is a direct consequence of the following proposition.
\begin{prop}
\label{p:ESC revisited}
Let~$f$ be a map in~$\sA$ that is topologically exact on~$J(f)$ and such that~$\chiinff > 0$.
Then for every periodic point~$p$ in the boundary of a Fatou component of~$f$, there are~$\chi > 0$ and~$\delta > 0$ such that
\begin{equation*}
\liminf_{n \to + \infty} \frac{1}{n} \ln \max \left\{ |W| : W \text{ connected component of~$f^{-n}(B(p, \delta))$} \right\}
<
- \chi.
\end{equation*}
\end{prop}
\begin{proof}
Fix~$\chi_0$ in~$(0, \chiinff)$, put~$\ell \= \prod_{c \in \Crit'(f)} \ell_c$, and denote by~$\pi \ge 1$ the period of~$p$ and by~$\cO$ the forward orbit of~$p$.
Let~$\delta_0 > 0$ and~$\varepsilon$ by given by Lemma~\ref{l:Koebe principle} with~$K = 2$.
Reducing~$\delta_0$ if necessary we assume~$f$ satisfies the conclusions of Lemma~\ref{l:pull-backs} with~$\delta_2 = \delta_0$.
Moreover, we assume that for every critical point~$c$ of~$f$, every interval~$\hJ$ contained in~$I$, and every pull-back~$\hW$ of~$J$ by~$f$ that is contained in~$B(c, \delta_0)$, the following property holds for every~$\kappa$ in~$(0, 1/2]$: If~$W$ is a pull-back of~$\kappa \hJ$ by~$f$ contained in~$\hW$, then
$$ (3 \kappa)^{-1/\ell_c} W \subset \hW. $$
We also assume there is~$M_0 > 1$ such that for every critical point~$c$ of~$f$ and every~$x$ in~$B(c, \delta_0)$ we have
$$ |Df(x)| \le M_0 |x - c|^{\ell_c - 1}. $$
Taking~$M_0$ larger if necessary we assume
$$ M_0 > \sup_I |Df|. $$

In view of Lemma~\ref{l:backward Lyapunov} and part~$1$ of the Main Theorem', there is~$\delta_*$ in~$(0, \delta_0)$ such that for every~$y$ in~$f^{-1}(\cO) \setminus \cO$, every integer~$m \ge 1$, and every pull-back~$W$ of~$B(y, \delta_*)$ by~$f^m$, we have
\begin{equation}
\label{e:strong rate shrinking}
|W| \le \min \left\{ \exp(- (m + 1) \chi_0), \delta_0 \right\}.
\end{equation}
Moreover, if we denote by~$U$ the Fatou component of~$f$ containing~$p$ in its boundary and denote by~$K$ the interval~$B(p, \delta_*) \setminus U$, then a similar property holds with~$B(y, \delta_*)$ replaced by~$K$: For every integer~$m \ge 1$ and every pull-back~$W$ of~$K$ by~$f^m$, we have~\eqref{e:strong rate shrinking}.
Note that the hypothesis~$\chiinff > 0$ implies that the periodic point~$p$ is hyperbolic repelling, so~$\chi_p(f) > 0$.
It follows that there is a constant~$\delta_{\dag} > 0$ and an integer~$\mu_{\dag} \ge 1$ such that for every~$\delta$ in~$(0, \delta_{\dag}]$, every integer~$m \ge 1$, and every pull-back~$W$ of~$B(p, \delta)$ by~$f^m$ containing a point~$p'$ of~$\cO$, we have
\begin{equation}
\label{e:strong ESC}
B(p', \delta \exp(- (m + \mu_{\dag}) \chi_p(f)))
\subset
W
\subset
B(p', \delta \exp(- (m - \mu_{\dag}) \chi_p(f))).
\end{equation}
Let~$\gamma_{\dag}$ in~$(0, 1)$ be sufficiently close to~$1$ so that
$$ \chi_{\dag} \= (1 - \gamma_{\dag}) \chi_0 < \chi_p(f)/\ell
\text{ and }
M_0^{1 - \gamma_{\dag}} < \exp(\gamma_{\dag} \chi_0 (\ell_c - 1) / 2). $$
Let~$n_{\dag} \ge 1$ be a sufficiently large integer so that
\begin{equation}
  \label{e:starting time 1}
  n_{\dag} \gamma_{\dag} \ge 2(\pi + 1),
\end{equation}
\begin{multline}
  \label{e:starting time 3}
12^{\# \Crit'(f) + 1} \exp(- n_{\dag} \gamma_{\dag} \chi_p(f) / (2 \ell))
\\ \le
\min \left\{ \varepsilon / 2, \exp( - (\pi + 2 \mu_{\dag}) \chi_{\dag}) \exp(- n_{\dag} \gamma_{\dag} \chi_{\dag}/2) \right\},
\end{multline}
\begin{multline}
  \label{e:starting time 2}
\left(  M_0^{1 - \gamma_{\dag}} \exp(- \gamma_{\dag} (\ell_c - 1) \chi_0 / 2) \right)^{n_{\dag}}
\\ <
\frac{1}{4} \left(2^{\ell_c - 1} \exp( (\pi + 2 \mu_{\dag}) (\ell_c - 1) \chi_0) \right)^{-1}.
\end{multline}
Reducing~$\delta_{\dag}$ if necessary, we assume that~$2 \delta_{\dag} \exp(\mu_{\dag} \chi_p(f)) < 1$ and that the following properties hold:
\begin{enumerate}[1.]
\item 
For every~$p'$ in~$\cO$, every pull-back of~$B(p', \delta_{\dag} \exp(\mu_{\dag} \chi_p(f)))$ by~$f$ that is disjoint from~$\cO$ contains a point~$y$ of~$f^{-1}(\cO)$ and is contained in~$B(y, \delta_*)$.
\item
For every~$n$ in~$\{1, \ldots, n_{\dag} \}$ and every pull-back~$W$ of~$B(p, \delta_{\dag})$ by~$f^n$, we have
$$ | W | \le \exp( - n \chi_{\dag}). $$
\end{enumerate}

To prove the proposition we show that for every integer~$n \ge 1$ and every pull-back~$W$ of~$B(p, \delta_{\dag})$ by~$f^n$, we have
$$ | W | \le \exp( - n \chi_{\dag}). $$
We proceed by induction in~$n$.
By our choice of~$\delta_{\dag}$ the desired assertion holds for each~$n$ in~$\{1, \ldots, n_{\dag} \}$.
Let~$n \ge n_{\dag} + 1$ be an integer for which this property holds with~$n$ replaced by~$n - 1$ and let~$W$ be a pull-back of~$B(x, \delta_{\dag})$ by~$f^n$.
If~$W$ intersects~$\cO$, then by~\eqref{e:strong ESC} with~$m$ replaced by~$n$ we have by our choice of~$\delta_{\dag}$,
$$ |W|
\le
2 \delta_{\dag}
\exp(- (n - \mu_{\dag}) \chi_p(f))
<
\exp(- n \chi_p(f)). $$
Since~$\chi_p(f) \ge \chiinff > \chi_0 > \chi_{\dag}$, this proves the induction hypothesis in this case.

Assume~$W$ does not intersect~$\cO$ and denote by~$m$ the largest element of~$\{0, \ldots, n - 1 \}$ such that the pull-back~$W_0$ of~$B(p, \delta_{\dag})$ by~$f^m$ containing~$f^{n - m}(W)$ intersects~$\cO$.
Let~$p_0$ be the point of~$W_0$ in~$\cO$.
By maximality of~$m$, the pull-back~$W'$ of~$B(p, \delta_{\dag})$ by~$f^{m + 1}$ containing~$f^{n - (m + 1)}(W)$ is disjoint from~$\cO$.
By our choice of~$\delta_{\dag}$ it follows that~$W'$ contains a point~$y$ of~$f^{-1}(\cO) \setminus \cO$ and is contained in~$B(y, \delta_*)$. 
So by our choice of~$\delta_*$ we have,
$$ |W| \le \exp(- (n - m) \chi_0). $$
In the case where~$m \le \gamma_{\dag} n$, we obtain
$$ |W|
\le
\exp(- n (1 - \gamma_{\dag}) \chi_0)
=
\exp(- n \chi_{\dag}), $$
so the induction hypothesis is verified in this case.

Suppose~$m \ge \gamma_{\dag} n$.
Note that by~\eqref{e:strong ESC} we have
$$ W_0 \subset B(p_0, \delta_{\dag} \exp(- (m - \mu_{\dag}) \chi_p(f))). $$
By~\eqref{e:strong ESC} and our choice of~$n_{\dag}$ there is an integer~$\whm$ satisfying
\begin{equation}
\label{e:middle time}
m/2 - 2 \mu_{\dag} - \pi
\le
\whm
\le
m/2 + \pi,
\end{equation}
such that~$\whm - m$ is divisible by~$\pi$, and such that the pull-back~$\hW_0$ of~$B(p, \delta_{\dag})$ by~$f^{\whm}$ containing~$p_0$ contains
$$ B(p_0, \delta_{\dag} \exp(- (m/2 - \mu_{\dag}) \chi_p(f))) $$
For each~$j$ in~$\{1, \ldots, n - m \}$ denote by~$W_j$ (resp.~$\hW_j$) the pull-back of~$W_0$ (resp.~$\hW_0$) by~$f^j$ containing~$f^{n - (m + j)}(W)$ and let~$\kappa_j > 0$ be the smallest number such that
$$ \kappa_j^{-1} W_j \subset \hW_j. $$
Note that~$\kappa_0 \le \exp(- m \chi_p(f) /2)$.
By~\eqref{e:starting time 1} and~\eqref{e:middle time} we have~$\whm + n - m \le n - 1$.
So by the induction hypothesis and~\eqref{e:middle time} we have
\begin{multline}
\label{e:inductive shrinking}
|\hW_{n - m}|
\le
\exp(- (\whm + n - m) \chi_{\dag})
\\ \le
\exp((\pi + 2 \mu_{\dag})\chi_{\dag}) \exp( - (n - m/2) \chi_{\dag}).
\end{multline}
So, in view of~\eqref{e:starting time 3}, to complete the proof of the induction step it is enough to prove
\begin{equation}
  \label{e:final space}
  \kappa_{n - m}
\le
12^{\# \Crit'(f) + 1} \kappa_0^{1/\ell}.
\end{equation}

To prove this inequality, we show that for each critical point~$c$ of~$f$ there are at most~$1$ element~$j$ of~$\{1, \ldots, n - m \}$ such that~$\hW_j$ contains~$c$.
Part~$1$ of Lemma~\ref{l:pull-backs} implies that for each~$j$ in~$\{1, \ldots, n - m \}$ the set~$\hW_j$ intersects~$J(f)$, so~$f^{\whm + j}(\hW_j)$ intersects~$K$.
Moreover, the pull-back~$\hK_j$ of~$K$ by~$f^{\whm + j}$ contained in~$\hW_j$ is an interval.
By our choice of~$\delta_*$, for every~$j \ge 1$ we have
\begin{equation}
\label{e:Julia shrinking}
  |\hK_j| \le \exp(- (\whm + j) \chi_0). 
\end{equation}
Suppose by contradiction there is a critical point~$c$ of~$f$ and elements~$j$ and~$j'$ of~$\{1, \ldots, n - m \}$ such that~$\hW_{j'}$ and~$\hW_j$ contain~$c$ and such that~$j' \ge j + 1$.
By part~$1$ of Lemma~\ref{l:pull-backs} this implies that~$c$ is in~$\hK_j$ and in~$\hK_{j'}$.
By~\eqref{e:Julia shrinking} this implies that~$f^{j' - j}(c)$ is in~$B(c, \exp(- \whm \chi_0))$.
By our choices of~$\delta_0$ and~$M_0$ and by~\eqref{e:middle time}, the derivative of~$f^{j' - j}$ on~$B(c, 2\exp(- \whm \chi_0))$ is bounded from above by
\begin{multline*}
 M_0^{j' - j} (2 \exp(- \whm \chi_0))^{\ell_c - 1}
\\ \le
2^{\ell_c - 1} \exp((\pi + 2 \mu_{\dag}) (\ell_c - 1) \chi_0) M_0^{(1 - \gamma_{\dag}) n}  \exp( - n \gamma_{\dag} (\ell_c - 1) \chi_0 / 2).  
\end{multline*}
By~\eqref{e:starting time 2} it follows that the derivative of~$f^{j' - j}$ on~$B(c, 2 \exp(- \whm \chi_0))$ is bounded from above by a number that strictly less than~$1/4$.
Since~$f^{j' - j}(c)$ is in~$B(c, 2 \exp(- \whm \chi_0))$, this implies that~$f^{j' - j}$ contains a hyperbolic attracting fixed point and that~$c$ converges to this point under forward iteration by~$f^{j' - j}$.
This contradicts the fact that~$c$ is in~$J(f)$.

To prove~\eqref{e:final space}, let~$k$ be the number of those~$j$ in~$\{1, \ldots, n - m \}$ such that~$\hW_j$ contains a critical point of~$f$.
If~$k = 0$, then~$f^{n - m}$ maps~$\hW_{n - m}$ diffeomorphically to~$\hW_0$.
Noting that by~\eqref{e:starting time 3} we have~$\kappa_0 \le \varepsilon /2$, by Lemma~\ref{l:Koebe principle} this implies~$\kappa_{n - m} \le 4 \kappa_0$ and thus~\eqref{e:final space}.

Suppose~$k \ge 1$.
Put~$j_0 \= 0$ and let~$j_1 <  \cdots < j_k$ be all the elements~$j$ of~$\{1, \ldots, n - m \}$ such that~$\hW_j$ contains a critical point of~$f$.
By part~$1$ of Lemma~\ref{l:pull-backs}, for each~$s$ in~$\{1, \ldots, k \}$ the set~$\hW_{j_s}$ contains a unique critical point of~$f$; denote it by~$c_s$.
Moreover, by our choice of~$\delta_*$ the set~$\hW_{j_s}$ is contained in~$B(c_s, \delta_0)$.
We prove by induction that for every~$s$ in~$\{0, \ldots, k \}$ we have
\begin{equation}
  \label{e:inductive space}
  \kappa_{j_s}
\le
12^s \kappa_0^{\left( \prod_{s' = 1}^{s} \ell_{c_{s'}} \right)^{-1}}.
\end{equation}
This is trivially true when~$s = 0$.
Let~$s$ in~$\{0, \ldots, k - 1 \}$ be such that this property holds for~$s$.
First note that by our choice of~$\delta_0$ we have
\begin{equation}
\label{e:critical pull-back}
\kappa_{j_{s + 1}}
\le
(3 \kappa_{j_{s + 1} - 1})^{1 / \ell_{c_{s + 1}}}
\le
3 \kappa_{j_{s + 1} - 1}^{1 / \ell_{c_{s + 1}}}.
\end{equation}
If~$j_{s + 1} = j_s + 1$, then the induction step follows from~\eqref{e:inductive space}.
If~$j_{s + 1} \ge j_s + 2$, then~$f^{j_{s + 1} - j_s - 1}$ maps~$\hW_{j_{s + 1} - 1}$ diffeomorphically to~$\hW_{j_s}$.
Noting that~\eqref{e:starting time 3} and~\eqref{e:inductive space} imply~$\kappa_{j_s} \le \varepsilon /2$, by Lemma~\ref{l:Koebe principle} we have~$\kappa_{j_{s + 1} - 1} \le 4 \kappa_{j_s}$.
Together with~\eqref{e:inductive space} and~\eqref{e:critical pull-back}, this completes the proof of the induction step.

By~\eqref{e:inductive space} with~$s = k$ we have~$\kappa_{j_k} \le 12^{\# \Crit'(f)} \kappa_0^{1/\ell}$.
This proves~\eqref{e:final space} if~$j_k = n - m$.
In the case where~$j_k \le n - m - 1$, using that~$f^{n - m - j_k}$ maps~$\hW_{n - m}$ diffeomorphically to~$\hW_{j_k}$ and that we have~$\kappa_{j_k} \le \varepsilon / 2$ by~\eqref{e:starting time 3}, by Lemma~\ref{l:Koebe principle} we have
$$ \kappa_{n - m}
\le
4 \kappa_{j_k}
\le
12^{\# \Crit'(f) + 1} \kappa_0^{1/\ell}. $$
This proves~\eqref{e:final space} and completes the proof of the induction step and of the proposition.
\end{proof}

\appendix
\section{General properties of smooth interval maps}
\label{s:general properties}
In this section we gather a few general facts of maps in~$\sA$, that are used in the proof of Proposition~\ref{p:ESC revisited}.

The following version of the Koebe principle follows from~\cite[Theorem~C ($2$)(ii)]{vStVar04}.
A periodic point~$p$ of period~$n$ of a map~$f$ in~$\sA$ is \emph{hyperbolic repelling} if~$|Df^n(p)| > 1$.
\begin{lemm}[Koebe principle]
\label{l:Koebe principle}
Let~$f : I \to I$ be an interval map in~$\sA$ all whose periodic points in~$J(f)$ are hyperbolic repelling.
Then there is~$\delta_0 > 0$ such that for every~$K > 1$ there is~$\varepsilon$ in~$(0, 1)$ such that the following property holds.
Let~$J$ be an interval contained in~$I$ that intersects~$J(f)$ and satisfies~$|J| \le \delta_0$.
Moreover, let~$n \ge 1$ be an integer and $W$ a diffeomorphic pull-back of~$J$ by~$f^n$.
Then for every~$x$ and~$x'$ in the unique pull-back of~$\varepsilon J$ by~$f^n$ contained in~$W$ we have
$$ K^{-1} \le |Df^n(x)| / |Df^n(x')| \le K. $$
\end{lemm}

\begin{lemm}
\label{l:backward Lyapunov}
Let~$f : I \to I$ be a multimodal map in~$\sA$ having all of its periodic points in~$J(f)$ hyperbolic repelling and that is essentially topologically exact on~$J(f)$.
Then, for every~$\kappa > 0$ there is~$\delta_1 > 0$ such that for every~$x$ in~$J(f)$, every integer~$n \ge 1$, and every pull-back~$W$ of~$B(x, \delta_1)$ by~$f^n$, we have~$|W| \le \kappa$.
\end{lemm}
This lemma is a direct consequence of part~$2$ of the following lemma.
\begin{lemm}
  \label{l:pull-backs}
Let~$f : I \to I$ be a multimodal map in~$\sA$ having all of its periodic points in~$J(f)$ hyperbolic repelling.
Then there is~$\delta_2 > 0$ such that for every~$x$ in~$J(f)$ the following properties hold.
\begin{enumerate}[1.]
\item 
For every integer~$n \ge 1$, every pull-back~$W$ of~$B(x, \delta_2)$ by~$f^n$ intersects~$J(f)$, contains at most~$1$ critical point of~$f$, and is disjoint from~$(\Crit(f) \cup \partial I) \setminus J(f)$.
\item
If in addition~$f$ is essentially topologically exact on~$J(f)$, then
$$ \lim_{n \to + \infty} \max \left\{ |W| : W \text{ connected component of~$f^{-n}(B(x, \delta_2))$} \right\}
=
0. $$
\end{enumerate}
\end{lemm}
In the proof of Lemma~\ref{l:pull-backs} below we use the fact that every Fatou component is mapped to a periodic Fatou component under forward iteration, see~\cite[Theorem~A']{MardMevSt92}.
We also use the fact that each interval map in~$\sA$ has at most a finite number of periodic Fatou components, see~\cite[\S$1$]{MardMevSt92}.

\begin{proof}[Proof of Lemma~\ref{l:pull-backs}]
The assertion in part~$1$ that~$W$ contains at most~$1$ critical point of~$f$ is a direct consequence of part~$2$.
In part~$1$ below we prove the rest of the assertions in part~$1$.
In part~$2$ below we complete the proof of the proposition by proving part~$2$.
Put
$$ \sS \= (\Crit(f) \cup \partial I) \setminus J(f). $$

\partn{1}
Let~$V$ be a periodic Fatou component of~$f$ and let~$p \ge 1$ be its period.
Then each boundary point~$y$ of~$V$ in~$J(f)$ is such that~$f^p(y)$ is in~$\partial V$.
This implies that~$f^p(y)$ is a periodic point of~$f$ in~$J(f)$ and our hypotheses imply that~$f^p(y)$ is hyperbolic repelling.
It follows that there is a compact interval~$K_V$ contained in~$V$, such that for every integer~$n \ge 0$ the set~$f^n(\sS)$ is disjoint from~$V \setminus K_V$.
Since~$f$ has at most a finite number of periodic Fatou components and since every Fatou component of~$f$ is mapped to a periodic Fatou component of~$f$ under forward iteration, it follows that there is~$\delta_* > 0$ such that for every integer~$n \ge 1$ the distance between~$f^n(\sS)$ and~$J(f)$ is at least~$\delta_*$.

To prove part~$1$ with~$\delta_2 = \delta_*$, let~$x$ be a point in~$J(f)$, let~$n \ge 1$ be an integer, and let~$W$ be a pull-back of~$B(x, \delta_*)$ by~$f^n$.
By definition of~$\delta_*$, the set~$W$ is disjoint from~$\sS$.
Put~$W_0 \= B(x, \delta_*)$ and for every~$j$ in~$\{1, \ldots, n \}$ let~$W_j$ be the pull-back of~$W_0$ by~$f^j$ that contains~$f^{n - j}(W)$.
We show by induction in~$j$ that~$W_j$ intersects~$J(f)$.
By definition~$W_0$ intersects~$J(f)$.
Let~$j$ be an integer in~$\{0, \ldots, n - 1 \}$ such that~$W_j$ intersects~$J(f)$ and suppose by contradiction that~$W_{j + 1}$ does not intersect~$J(f)$.
Since~$W_{j + 1}$ is disjoint from~$\sS$, it follows that~$f$ is injective on~$W_{j + 1}$ and that~$f(W_{j + 1}) = W_j$.
Since by hypothesis~$J(f)$ is completely invariant, this implies that~$W_j$ is disjoint from~$J(f)$ as well.
We obtain a contradiction that completes the proof of the induction hypothesis and of the fact that for each~$j$ in~$\{0, \ldots, n \}$ the set~$W_j$ intersects~$J(f)$.
Since~$W_n = W$, this proves that~$W$ intersects~$J(f)$.

\partn{2}
Let~$\delta_*$ be as in part~$1$.
Our hypothesis that all the periodic points of~$f$ in~$J(f)$ are hyperbolic repelling implies that there are~$\delta_{\dag}$ in~$(0, \delta_*)$ and~$\gamma$ in~$(0, 1)$ such that for every periodic point~$y$ in the boundary of a Fatou component the following property holds: For every integer~$n \ge 1$ the pull-back~$W$ of~$B(f^n(y), \delta_{\dag})$ by~$f^n$ that contains~$y$ satisfies~$|W| < \gamma^n$.

Since~$f$ is essentially topologically exact on~$J(f)$, there is an interval~$I_0$ contained in~$I$ that contains all critical points of~$f$ and such that the following properties hold: $f(I_0) \subset I_0$, $f|_{I_0}$ is topologically exact on~$J(f|_{I_0})$, and~$\bigcup_{n = 0}^{+ \infty} f^{-n}(I_0)$ contains an interval whose closure contains~$J(f)$.
Notice in particular that~$J(f|_{I_0})$ is not reduced to a point.
Reducing~$\delta_{\dag}$ if necessary we assume
\begin{equation}
  \label{e:BLS size}
\delta_{\dag}
<
\diam (J(f|_{I_0})),
\end{equation}
\begin{equation}
  \label{e:BLS size 2}
\delta_{\dag}
<
\min \{ |V| : V \text{ periodic Fatou component of~$f$} \}.
\end{equation}

To prove part~$2$ with~$\delta_2 = \delta_{\dag} / 2$, we proceed by contradiction.
If the desired assertion does not hold, then there is~$\kappa > 0$, a sequence of positive integers~$(n_j)_{j = 1}^{+ \infty}$ such that~$n_j \to + \infty$ as~$j \to + \infty$, a sequence of points~$(x_j)_{j = 1}^{+ \infty}$ in~$J(f)$, and a sequence of intervals~$(J_j)_{j = 1}^{+ \infty}$ such that for every~$j$ we have
\begin{equation}
  \label{e:definite to small}
  |J_j| \ge \kappa,
x_j \in f^{n_j}(J_j),
\text{ and }
f^{n_j}(J_j) \subset B(x_j, \delta_{\dag}/2).
\end{equation}

Note that by part~$1$ for each~$j \ge 1$ the interval~$J_j$ intersects~$J(f)$; denote by~$K_j$ the convex hull of~$J_j \cap J(f)$.
In part~$2.1$ we prove that for every sufficiently large~$j$ we have~$|K_j| \ge \kappa / 3$ and in part~$2.2$ we use this fact to complete the proof of part~$2$.

\partn{2.1}
If~$J(f) = I$, then for each~$j$ we have~$K_j = J_j$ and therefore~$|K_j| \ge \kappa$, so there is nothing to prove in this case.
Assume~$J(f) \neq I$ and let~$P \ge 1$ be the largest period of a periodic Fatou component of~$f$.
Given a Fatou component~$V$ of~$f$, let~$n(V)$ be the least integer~$n \ge 0$ such that the Fatou component of~$f$ containing~$f^n(V)$ is periodic.
Clearly, $|V| \to 0$ as~$n(V) \to + \infty$.
Let~$N_0 \ge 1$ be such that for every Fatou component~$V$ satisfying~$n(V) \ge N_0$ we have~$|V| \le \kappa/3$.
Let~$\delta_{\ddag} > 0$ be sufficiently small so that for every~$x$ in~$J(f)$, every~$n$ in~$\{0, \ldots, N_0 + P \}$, and every pull-back~$W$ of~$B(x, \delta_{\ddag})$ by~$f^n$, we have~$|W| \le \kappa /3$.
Finally, let~$N_1 \ge 1$ be such that~$\gamma^{N_1} \le \delta_{\ddag}$.

We prove that for every~$j$ such that~$n_j \ge P + N_0 + N_1$ we have~$|K_j| \ge \kappa / 3$.
It is enough to show that for every connected component~$U$ of~$J_j \setminus K_j$ we have~$|U| \le \kappa / 3$.
Let~$V$ be the Fatou component of~$f$ containing~$U$.
If~$n(V) \ge N_0$, then~$|U| \le \kappa/3$.
It remains to consider the case where~$n(V) \le N_0 - 1$.
Then~$f^{n(V)}(J_j)$ intersects a periodic Fatou component, so it must contain one of its boundary points.
By the definition of~$P$ this implies that~$f^{n(V) + P}(J_j)$ contains a periodic point~$y$ in the boundary of a Fatou component.
Noting that by our choice of~$j$ we have~$n_j - P - n(V) \ge N_1$, by~\eqref{e:definite to small} we conclude that
$$ f^{n_j - P - n(V)}(y)
\in
f^{n_j}(J_j)
\subset
B(x_j, \delta_{\dag}/2)
\subset
B(f^{n_j - P - n(V)}(y), \delta_{\dag}). $$
Using the definition of~$\delta_{\dag}$, $\gamma$, and~$N_1$, we obtain
$$ |f^{P + n(V)}(J_j)|
\le
\gamma^{n_j - P - n(V)}
\le
\gamma^{N_1}
\le
\delta_{\ddag}. $$
Since~$n(V) \le N_0 - 1$, by definition of~$\delta_{\ddag}$ we have~$|U| \le |J_j| \le \kappa/3$.
This completes the proof that for every sufficiently large~$j$ we have~$|K_j| \ge \kappa /3$.

\partn{2.2}
Taking subsequences if necessary we assume~$(K_j)_{j = 1}^{+\infty}$ converges to an interval~$K$.
We have~$|K| \ge \kappa/3$ and~$\partial K \subset J(f)$.

Suppose the interior of~$K$ intersects~$J(f)$.
Then there is a compact interval~$K'$ contained in the interior of~$K$ and an integer~$n \ge 1$ such that~$f^n(K') \supset J(f|_{I_0})$.
This implies that for every sufficiently large~$j$ we have~$f^{n_j}(K_j) \supset J(f|_{I_0})$ and therefore by~\eqref{e:BLS size} we have
$$ |f^{n_j}(J_j)| \ge |f^{n_j}(K_j)| \ge \diam(J(f|_{I_0})) > \delta_{\dag}. $$
We get a contradiction with~\eqref{e:definite to small}.

It remains to consider the case where the interior of~$K$ is contained in~$F(f)$.
Then the interior~$V$ of~$K$ is a Fatou component of~$f$.
Since for each~$j$ we have~$\partial K_j \subset J(f)$, for large~$j$ the interval~$K_j$ contains~$V$.
Taking~$j$ larger if necessary we assume~$n_j \ge n(V)$.
If~$V$ contains no turning point of~$f^{n_j}$, then~$f^{n_j}(V)$ is a periodic Fatou component of~$f$ and therefore by~\eqref{e:BLS size 2}
$$ |f^{n_j}(J_j)| \ge |f^{n_j}(K_j)| \ge |f^{n_j}(V)| > \delta_{\dag}. $$
We thus obtain a contradiction with~\eqref{e:definite to small}.
It follows that~$V$ contains a turning point of~$f^{n_j}$.
Using that~$K_j$ intersects~$J(f)$, we conclude by the definition of~$\delta_*$ that
$$ |f^{n_j}(J_j)| \ge |f^{n_j}(K_j)| \ge \delta_* > \delta_{\dag}. $$
We obtain a contradiction that completes the proof of part~$2$ and of the proposition.
\end{proof}

% For Bibtex bibliography:
\bibliographystyle{alpha}
% \bibliography{$HOME/papers/0BIB/papers}

\end{document}